\documentclass[15pt]{article}

\usepackage{amsmath,amsfonts,amssymb,amsthm,amscd}

\newcounter{stepctr}
{\end{list}}

\newtheorem{thm}{Theorem}[section]
\newtheorem{prop}[thm]{Proposition}
\newtheorem{cor}[thm]{Corollary}

\theoremstyle{definition}
\newtheorem{dfn}[thm]{Definition}

\newtheorem{rema}[thm]{Remark}
\newtheorem{lem}[thm]{Lemma}
\newtheorem{prob*}{Open problem}
\newcommand{\demo}{\begin{proof}}

\newcommand{\R}{\ensuremath{\mathcal{R}}}

\newcommand{\N}{\mathbb{N}}

\newcommand{\C}{\mathbb{C}}

\newcommand{\ind}[1]{{\rm ind}\,({#1})}

\def\ll^2{{\mathcal L}(\ell^2(\N))}

\def\f^0x{{\mathcal F^0}(X) }

\title
{\bf   On the index of pseudo B-Fredholm operator}

\author{ Zakariae  Aznay, Abdelmalek Ouahab, Hassan Zariouh }

\date{}
\begin{document}

\maketitle \thispagestyle{empty}

\begin{abstract}\noindent\baselineskip=10pt
The index of a pseudo B-Fredholm operator will be defined and generalize the usual index of a B-Fredholm operator. This concept will be used to extend some known results  in Fredholm's theory. Among other results, the nullity, the deficiency, the ascent and the descent will be extended and defined for a pseudo-Fredholm operator.

\end{abstract}

 \baselineskip=15pt
  \baselineskip=15pt

\section{Introduction}
\par Let $T \in L(X);$ where   $L(X)$ is the  Banach algebra of bounded linear operators acting on an infinite dimensional complex Banach space $X.$   $\mathcal{N}(T)$ and   $\R(T)$  are   respectively  the kernel and  the range  of $T.$   $T$ is said to be upper semi-Fredholm, if $\R(T)$ is closed and $\mbox{dim}\,\mathcal{N}(T) <\infty,$ while $T$ is called lower semi-Fredholm, if  $\mbox{codim}\,\R(T) < \infty.$ If $T$ is an upper or a lower semi-Fredholm then is called a semi-Fredholm operator and its index   is defined by $\ind T = \mbox{dim}\,\mathcal{N}(T) -\mbox{codim}\,\R(T).$ $T$ is called a  Fredholm operator  if it  is a  semi-Fredholm with an integer index.
\par\noindent  A  subspace $M$ of $X$ is $T$-invariant if $T(M)\subset M$ and in this case  $T_{M}$ means the restriction of $T$ on $M.$  We say that  $T$   is completely reduced   by  a pair    $(M,N)$   ($(M,N) \in Red(T)$ for brevity) if $M$ and  $N$ are  closed $T$-invariant subspaces  of $X$ and  $X=M\oplus N;$   here  $M\oplus N$ means that $M\cap N=\{0\}.$ Let $(M,N) \in Red(T)$ and let the  list of  the following points:
\vspace{0.8em}
\par
(i) $T_{M}$ is semi-regular (i.e, $\R(T_{M})$ is closed and $\mathcal{N}(T_{M})\subset \bigcap_{n\in \N}\R(T^{n}_{M})$).

(i') $T_{M}$ is semi-Fredholm.

(ii) $T_{N}$ is nilpotent of degree $d$ (i.e, $T^{d}_{N}=0$ and  $T^{d-1}_{N}\neq 0$).

(ii') $T_{N}$ is quasi-nilpotent (i.e, $0$ is the only point of the spectrum  $\sigma(T_{N})$ of $T_{N}$).
\vspace{0.8em}
\par\noindent In \cite[Theorem 4]{kato},  T. kato proved that if $T$ is a semi-Fredholm operator, then there exists $(M,N) \in Red(T)$  satisfies the points (i) and (ii) listed above;  this decomposition $(M, N)$  is called   Kato's decomposition associated to $T.$  In the case of $X$ is a Hilbert space, J. P. Labrousse \cite{labrousse} studied and characterized  the operators which admit such a decomposition and he called them  quasi-Fredholm operators of degree $d.$ The class of quasi-Fredholm operators has been extended by Mbekhta and Muller in \cite[p. 143]{mbekhta-muller}, Poon in \cite{poon} to the general case of Banach space operators.  \\
In 1999 and 2001, Berkani and Sarih \cite{berkani, berkani-sarih} generalized the concept of  semi-Fredholm operators  to a class  called semi-B-Fredholm operators as follows: For $n\in\N,$  let $T_{[n]}: \mathcal{R}(T^n)\rightarrow \mathcal{R}(T^n)$  be the restriction of $T$ to  $\mathcal{R}(T^n)$ viewed as a map from $\R(T^n)$ into $\R(T^n)$ (in particular, $T_{[0]} = T$).
  $T$ is said to be semi-B-Fredholm if for some integer $n \geq 0$ the range $\mathcal{R}(T^n)$
is closed and $T_{[n]}$ is semi-Fredholm. And  in this case, they  showed in \cite[Proposition 2.1]{berkani-sarih} that $\mathcal{R}(T^m)$
is closed,  $T_{[m]}$ is a semi-Fredholm operator. Moreover, we show in   Proposition \ref{propaznayzariouh} below that  $\mbox{ind}(T_{[m]})=\mbox{ind}(T_{[n]}),$ for each  $ m \geq n.$ This  defines  the index of a semi-B-Fredholm operator $T$ as the index of the semi-Fredholm operator $T_{[n]}.$  In particular, if $T$ is a semi-Fredholm operator we find   the usual  definition of the   index of $T.$ Furthermore,  in the case of $X$ is a Hilbert space they showed that $T$ is semi-B-Fredholm if and only if there exists $(M,N) \in Red(T)$ such  that $T$ satisfies the points (i') and (ii) defined above, see \cite[Theorem 2.6]{berkani-sarih}. Note that  the B-Fredholm operators are also characterized in the general case of Banach spaces  following this last decomposition, see \cite[Theorem 7]{muller}.\\
The concept of quasi-Fredholm operators  has been generalized by M. Mbekhta  \cite{mbekhta}  with  replacing the  point (ii)  by the point (ii') in the  definition of a quasi-Fredholm operator; these operators are named pseudo-Fredholm and the decomposition $(M, N)$ associated to a pseudo-Fredholm $T$ called  generalized kato decomposition ($(M, N) \in GKD(T)$ for brevity).  In the same way as the generalization of the notion of quasi-Fredholm operators to the notion of semi-B-Fredholm operators \cite{berkani, berkani-sarih}, the notion of pseudo-Fredholm operators has been generalized to the notion called pseudo semi-B-Fredholm \cite{amouch, boasso, rwassa, tajmouati1, zariouh-zguitti} by replacing the point (i) by the point (i'). Moreover, the authors \cite{berkani, berkani-sarih} have well defined the index of a semi-B-Fredholm operator as a natural extension of the useul index of a semi-Fredholm operator. In a natural way, we ask the following question: can we define the index of a pseudo semi-B-Fredholm operator?
\par The main purpose of the present   paper is to answer affirmatively to this question. For a given pseudo semi-B-Fredholm operator $T,$ we define its   index  as the index of the semi-Fredholm operator  $T_{M};$   where   $(M,N) \in Red(T)$ such that $T_{M}$ is semi-Fredholm and $T_{N}$ is quasi-nilpotent. Furthermore, we prove that this definition of the index of $T$  is independent of   the choice of the decomposition   $(M, N)$ of $T.$ In particular, in the case of $T$ is a semi-Fredholm or a B-Fredholm   or  a Hilbert space semi-B-Fredholm operator, we  find   the usual  definition of the   index of $T$ as semi-Fredholm or a B-Fredholm or a Hilbert space semi-B-Fredholm operator. Using this notion of the index for pseudo semi-B-Fredholm operator, we generalize some known  properties  in the index theory for Fredholm  and B-Fredholm operators.\\
As an application of the results obtained in this paper, we prove that if  $T\in L(X)$ is  a B-Fredholm operator,  then $\R(T^{*})+\mathcal{N}(T^{*d})$ is closed in $\sigma(X^{*}, X)$ and $T^{*}$ is  a  B-Fredholm operator  with     $\mbox{ind}(T)=-\mbox{ind}(T^{*});$ where $d=\mbox{dis}(T)$ is the degree of stable iteration of $T$ and  $\sigma(X^{*}, X)$ is the weak*-topology on $X^{*}.$

\section{Index of pseudo semi-B-Fredholm operator}
We begin this section with the following lemma which will be useful in everything that follows.
\begin{lem}\label{lem0} Let $T \in L(X).$  If there exist   two pair of  closed $T$-invariant subspaces $(M, N),$ $(M^{'}, N^{'})$ such that  $ M \oplus  N=M^{'}\oplus N^{'} $ is closed, $T_{M}$ and  $T_{M^{'}}$ are semi-Fredholm operators, $T_{N}$  and $T_{N^{'}}$  are quasi-nilpotent operators, then $\mbox{ind}(T_{M})= \mbox{ind}(T_{M^{'}}).$
\end{lem}

\begin{proof}
Since $T_{M}$ and  $T_{M^{'}}$ are  semi-Fredholm operators,   then from punctured neighborhood theorem for semi-Fredholm operators, there  exists $\epsilon >0$ such that $B(0, \epsilon) \subset \sigma_{sf}(T_{M})^C\cap \sigma_{sf}(T_{M^{'}})^C,$  $\mbox{ind}(T_{M}- \lambda I) =\mbox{ind}(T_{M})$ and   $\mbox{ind}(T_{M^{'}}- \lambda I) =\mbox{ind}(T_{M^{'}})$ for every $\lambda \in B(0, \epsilon);$ where $\sigma_{sf}(.)$ is the semi-Fredholm spectrum.  As
$T_{N}$ and  $T_{N^{'}}$ are  quasi-nilpotent, then $B_{0}:=B(0, \epsilon)\setminus\{0\}\subset  \sigma_{sf}(T_{M})^C\cap \sigma_{sf}(T_{M^{'}})^C\cap  \sigma(T_{N})^C\cap \sigma(T_{N^{'}})^C\subset\sigma_{sf}(T_{M\oplus N})^C.$  Let $\lambda \in B_{0},$  as     $\mbox{ind}(T_{M}- \lambda I)+\mbox{ind}(T_{N} - \lambda I)=\mbox{ind}(T_{M^{'}}- \lambda I)+\mbox{ind}(T_{N^{'}} - \lambda I),$ then $\mbox{ind}(T_{M})=\mbox{ind}(T_{M^{'}}).$
\end{proof}

\begin{dfn}\cite{amouch, boasso, rwassa, tajmouati1, zariouh-zguitti}\label{dfn0}    $T \in L(X)$ is  said to be\\
(a) An upper   pseudo semi-B-Fredholm (resp., a lower pseudo semi-B-Fredholm, a pseudo B-Fredholm)  operator if there exists $(M, N) \in Red(T)$ such that $T_{M}$ is  an upper   semi-Fredholm (resp., a lower semi-Fredholm, a Fredholm)  operator  and $T_{N}$ is quasi-nilpotent. \\
(b) An upper    pseudo semi-B-Weyl (resp., a lower  pseudo semi-B-Weyl, a pseudo-B-Weyl)  operator if there exists $(M, N) \in Red(T)$ such that $T_{M}$ is  an upper   semi-Weyl (resp., a lower semi-Weyl, a Weyl)  operator  and $T_{N}$ is quasi-nilpotent. \\
(c) A pseudo semi-B-Fredholm (resp.,   pseudo semi-B-Weyl) operator  is an  upper   pseudo semi-B-Fredholm (resp., an upper   pseudo semi-B-Weyl) or  a lower pseudo semi-B-Fredholm  (resp., a lower  pseudo semi-B-Weyl) operator.
\end{dfn}
Let us denote by $\sigma_{upbf}(T),$ $\sigma_{lpbf}(T),$ $\sigma_{spbf}(T),$ $\sigma_{pbf}(T),$ $\sigma_{upbw}(T),$ $\sigma_{lpbw}(T),$ $\sigma_{spbw}(T)$ and  $\sigma_{pbw}(T)$  respectively, the upper pseudo semi-B-Fredholm spectrum, the  lower pseudo semi-B-Fredholm spectrum, the  pseudo semi-B-Fredholm spectrum, the pseudo B-Fredholm spectrum,  the upper pseudo semi-B-Weyl spectrum, the  lower  pseudo semi-B-Weyl spectrum, the   pseudo semi-B-Weyl spectrum and  the  pseudo-B-Weyl spectrum of a given $T \in L(X).$
\begin{dfn}\label{dfn1} Let $T \in L(X)$ be a pseudo semi-B-Fredholm operator.  We define the index $\mbox{ind}(T)$ of $T$ as the index of  $T_{M};$ where $M$ is a closed $T$-invariant subspace which    has a complementary closed $T$-invariant subspace  $N$  with $T_{M}$ is    semi-Fredholm  and $T_{N}$ is quasi-nilpotent. From Lemma \ref{lem0}, it is  clear that the index of $T$ is independent of the choice of the pair  $(M, N)$ appearing in the definition of the pseudo semi-B-Fredholm  $T$ (see  Definition \ref{dfn0}).
\end{dfn}
As a consequence of the  notion of the index of a pseudo semi-B-Fredholm operator, we deduce the following remark.
\begin{rema}\label{rema0}(i) Every quasi-nilpotent operator is a pseudo B-Fredholm  operator and its index equal to zero. And every semi-Fredholm operator is also a    pseudo semi-B-Fredholm,  and its usual  index  as a semi-Fredholm  coincides with its index as a pseudo semi-B-Fredholm  operator.\\
(ii) $T \in L(X)$ is an   upper pseudo semi-B-Weyl (resp., a lower  pseudo semi-B-Weyl, a  pseudo B-Weyl) operator   if and only if   $T$  is an upper pseudo semi-B-Fredholm  (resp., a lower  pseudo semi-B-Fredholm, a  pseudo B-Fredholm) with $\mbox{ind} (T)\leq 0$ (resp.,  $\mbox{ind} (T)\geq 0$,  $\mbox{ind} (T)=0$). \\
(iii)   If  $T \in L(X)$ and $S \in L(Y)$  are    pseudo semi-B-Fredholm operators, then $T\oplus S$ is  pseudo semi-B-Fredholm and   $\mbox{ind} (T\oplus S)=\mbox{ind} (T)+ \mbox{ind} (S).$
\end{rema}

\begin{prop}\label{prop0}    Let $T \in L(X).$ The following statements hold.\\
(i) $T$  is a   pseudo B-Fredholm   if and only if $T$ is  an  upper  and   lower pseudo semi-B-Fredholm.\\
(ii) $T$  is a   pseudo B-Weyl   if and only if $T$  is an  upper and   lower pseudo semi-B-Weyl.\\
(iii) $T$  is a   pseudo B-Fredholm   if and only if $T$ is  a  pseudo semi-B-Fredholm with $\mbox{ind}(T) \in \mathbb{Z}.$
\end{prop}
\begin{proof} (i) Suppose that  $T$ is  an upper  and  lower pseudo semi-B-Fredholm. Then   there exist  $(M, N),$ $(M^{'}, N^{'})\in Red(T)$  such that $T_{M}$ is an upper  semi-Fredholm, $T_{M^{'}}$ is a lower  semi-Fredholm, $T_{N}$ and  $T_{N^{'}}$ are quasi-nilpotent.   From Lemma \ref{lem0} we have $\mbox{ind}(T)=\mbox{ind}(T_{M})=\mbox{ind}(T_{M^{'}}),$ and so  $(\alpha(T_{M})+\beta(T_{M^{'}}))-\alpha(T_{M^{'}})=\beta(T_{M})\geq 0.$ Thus   $T_{M}$ and $T_{M^{'}}$  are   Fredholm operators. The converse is obvious.\\
(ii) Is a consequence of the first point.  The point  (iii) is obvious.
\end{proof}
From Proposition \ref{prop0} we obtain the following corollary.
\begin{cor}\label{cor0} For every  $T \in L(X),$ we have  $\sigma_{pbf}(T)=\sigma_{upbf}(T)\cup\sigma_{lpbf}(T)$ and $\sigma_{pbw}(T)=\sigma_{upbw}(T)\cup\sigma_{lpbw}(T).$
\end{cor}
The following proposition extends  \cite[Proposition 3.7.1]{laursen-neumann} to pseudo B-Fredholm operators.
\begin{prop} Let $T \in L(X),$ and let $A\subset X$ be a closed $T$-invariant   subspace of finite codimension. If $T$ is a pseudo B-Fredholm operator then $T_{A}$ is also a pseudo B-Fredholm operator and  in this case, $\mbox{ind}(T)=  \mbox{ind}(T_{A}).$ The converse is true  if $A$ has  a complementary  $T$-invariant  subspace.
\end{prop}
\begin{proof} Suppose that $T$ is a pseudo B-Fredholm, then there exists $(M, N) \in Red(T)$ such that $T_{M}$ is  Fredholm   and $T_{N}$ is quasi-nilpotent. Since $A$ is a closed $T$-invariant  subspace, then  $(M \cap A, N\cap A) \in Red(T_{A}).$ Moreover, from  \cite[Lemma 2.2]{kaashoek}  we have $\mbox{codim}_{M}(M\cap A):=\mbox{dim}\,\frac{M}{M \cap A}=\mbox{dim}\,\frac{A+M}{A}\leq \mbox{codim} \,A<\infty.$ Then we get  from \cite[Proposition 3.7.1]{laursen-neumann} that $T_{M \cap A}$ is a   Fredholm operator and  $\mbox{ind}(T_{M})=  \mbox{ind}(T_{M \cap A}).$ As $T_{N}$ is quasi-nilpotent then  $T_{N\cap A}$ is also a quasi-nilpotent operator. Consequently, $T_{A}$ is a pseudo B-Fredholm operator and  $\mbox{ind}(T)=  \mbox{ind}(T_{A}).$  Conversely, suppose that  $T_{A}$ is a pseudo B-Fredholm operator. Then  there exists $(M, N) \in Red(T_{A})$ such that $T_{M}$ is a Fredholm operator  and $T_{N}$ is quasi-nilpotent.  Since by hypotheses,   $A$ is  a  closed $T$-invariant   subspace of finite codimension and has   a complementary  $T$-invariant   subspace  $F$,   then  $(M\oplus F, N)\in Red(T).$ As  $F$ is of finite dimension, then $T_{F}$ is  Weyl and so   $T_{M\oplus F}$ is  Fredholm. Hence, $T$ is pseudo B-Fredholm and $\mbox{ind}(T)=\mbox{ind}(T_{M\oplus F})=\mbox{ind}(T_{M})+\mbox{ind}(T_{F})=\mbox{ind}(T_{M})=\mbox{ind}(T_{A}).$
\end{proof}
\begin{prop}\label{prop} Let $T \in L(X)$ be  a  pseudo semi-B-Fredholm operator.  Then for every  strictly positive integer $n,$ the operator   $T^{n}$ is  pseudo semi-B-Fredholm and $\mbox{ind}(T^{n})=n.\mbox{ind}(T).$
\end{prop}
\begin{proof} Since $T$  is a  pseudo semi-B-Fredholm, then  there exists $(M, N) \in Red(T)$ such that $T_{M}$ is a  semi-Fredholm operator  and $T_{N}$ is quasi-nilpotent. So   $(M, N) \in Red(T^{n}),$  $T_{M}^{n}$ is a  semi-Fredholm operator  and $T_{N}^{n}$ is quasi-nilpotent. As it is well known that   $\mbox{ind}(T_{M}^{n})=n.\mbox{ind}(T_{M})$ then   $\mbox{ind}(T^{n})=n.\mbox{ind}(T).$
\end{proof}
Let   $T \in L(X)$ and let  $$\Delta(T):=\{m\in\N \,:\, \R(T^{m})\cap \mathcal{N}(T)=\R(T^{r})\cap \mathcal{N}(T),\,\forall r \in \N  \,  \, r\geq m \}.$$
The degree of stable iteration $\mbox{dis}(T)$ of $T$ is defined as $\mbox{dis}(T)=\mbox{inf}\Delta(T);$ with the infimum taken  $\infty$ in the case of empty set, see \cite{aiena,labrousse}. We say that $T$ is semi-regular if $\R(T)$ is closed and $\mbox{dis}(T)=0.$\\
Let  $r \in \N.$  It is easily  seen that $r\geq \mbox{dis}(T)$ if and only if  $\R(T)+ \mathcal{N}(T^{m})=\R(T)+ \mathcal{N}(T^{r}),\,\forall m \in \N  \,\text{  such that}  \, m\geq r.$
\begin{dfn}\cite{mbekhta} An operator $T \in L(X)$ is said   pseudo-Fredholm   if there exists $(M, N) \in Red(T)$ such that  $T_{M}$ is a semi-regular  operator  and $T_{N}$ is quasi-nilpotent. In this case, we say that  the pair $(M, N)$ is a generalized Kato decomposition associated to $T,$  and we write $(M,N) \in GKD(T)$ for brevity.
\end{dfn}
The next proposition  gives a characterization of pseudo semi-B-Fredholm operators.
\begin{prop}\label{prop1} Let $T \in L(X).$  $T$ is pseudo semi-B-Fredholm if and only if $T=T_{1}\oplus T_{2};$ where $T_{1}$ is a semi-Fredholm  and   semi-regular operator and $T_{2}$ is  quasi-nilpotent. In particular, a pseudo semi-B-Fredholm is  pseudo-Fredholm.
\end{prop}
\begin{proof} Let  $T$ be a  pseudo semi-B-Fredholm operator, and let $(M,N)\in GKD(T).$  From \cite[Theorem 4]{kato}, there exists $(A, B)\in GKD(T_{M})$ such that   $\mbox{dim}\,B<\infty.$  Since $T_{M}$ is semi-Fredholm, then $T_{A}$ is semi-Fredholm. On the other hand, it is easy to get   $(A, B\oplus N) \in GKD(T).$ The converse is obvious.
\end{proof}
Let $T \in L(X)$ be  a pseudo semi-B-Fredholm operator. According to Proposition \ref{prop1}, we focus in the sequel only on the  pairs  $(M,N) \in GKD(T)$ such that  $T_{M}$  is  semi-Fredholm. We  denote in the sequel \cite{mbekhta} by: the analytic core  and the  quasi-nilpotent part of $T$ defined respectively, by
$$\mathcal{K}(T)=\{x \in X : \exists \epsilon>0  \text{ and } \exists (u_{n})_{n} \subset X \text{ such that } x=u_{0}, Tu_{n+1}=u_{n} \text{ and } \|u_{n}\|\leq \epsilon^{n}\|x\| \,  \forall n \in \N\}$$
$$ \text{ and }\mathcal{H}_{0}(T)=\{x \in X : \lim\limits_{n \rightarrow \infty}\,\|T^{n}x\|^{\frac{1}{n}}=0\}.$$\\
For the sake of completeness, and to give to the reader a good overview of the subject, we include here the following  proposition.
\begin{prop}\label{propberkanisarih}{\rm\cite[Proposition 2.1]{berkani-sarih}}
Let $T\in L(X).$  If there exists an integer $n \in \N$ such that $\R(T^{n})$ is closed and such that the operator $T_{[n]}$ is an upper semi-Fredholm (resp. a lower semi-Fredholm) operator, then $\R(T^m)$ is closed, $T_{[m]}$ is an upper semi-Fredholm (resp. a lower semi-Fredholm) operator, for each $ m \geq n.$ Moreover, if $T_{[n]}$ is a Fredholm operator, then $T_{[m]}$ is a Fredholm operator and  $\mbox{ind}(T_{[m]})=\mbox{ind}(T_{[n]}),$ for each $ m\geq n.$
\end{prop}
Now, we prove in the following proposition that  if $T$ is a semi-B-Fredholm operator then   $\mbox{ind}(T_{[m]})=\mbox{ind}(T_{[n]}),$ for each $ m\geq n;$ where $n$ is any integer   such that $\R(T^{n})$ is closed and  $T_{[n]}$ is  semi-Fredholm.
\begin{prop}\label{propaznayzariouh}
Let $T\in L(X).$  If there exists an integer $n \in \N$ such that $\R(T^{n})$ is closed and  $T_{[n]}$ is  semi-Fredholm then $\R(T^m)$ is closed, $T_{[m]}$ is  semi-Fredholm and  $\mbox{ind}(T_{[m]})=\mbox{ind}(T_{[n]}),$ for each $ m \geq n.$
\end{prop}
\begin{proof}
 Suppose that   there exists an integer $n \in \N$ such that $\R(T^{n})$ is closed and such that the operator $T_{[n]}$ is  semi-Fredholm.  The first part of this proposition is proved in Proposition \ref{propberkanisarih}. Let us to show that $\mbox{ind}(T_{[m]})=\mbox{ind}(T_{[n]}),$ for each $ m \geq n.$
From  \cite[Lemma 3.2]{kaashoek} we have $\beta(T_{[n]})=\mbox{dim}\frac{X}{\R(T)+\mathcal{N}(T^{n})}.$  Moreover, from  \cite[Lemma 2.2]{kaashoek}  we have $\frac{\R(T)+\mathcal{N}(T^{n+1})}{\R(T)+\mathcal{N}(T^{n})} \cong \frac{\mathcal{N}(T_{[n]})}{\mathcal{N}(T_{[n+1]})}.$   We then obtain from \cite[Lemma 2.1]{taylor} that  $\beta(T_{[n]})=\beta(T_{[n+1]})+k_n(T)$ and $\alpha(T_{[n]})=\alpha(T_{[n+1]})+k_n(T);$ where $k_n(T):=\mbox{dim}\frac{\mathcal{N}(T_{[n]})}{\mathcal{N}(T_{[n+1]})}.$ It is easily seen that $k_n(T) \leq \min\{\alpha(T_{[n]}),\beta(T_{[n]})\}.$ Since  $T_{[n]}$ is  semi-Fredholm, then $k_n(T)$ is finite. Hence $\mbox{ind}(T_{[n]})=\alpha(T_{[n]})-\beta(T_{[n]})=\alpha(T_{[n]})-k_n(T)-(\beta(T_{[n]})-k_n(T))=\alpha(T_{[n+1]})-\beta(T_{[n+1]})=\mbox{ind}(T_{[n+1]}).$ It then follows by induction that $\mbox{ind}(T_{[m]})=\mbox{ind}(T_{[n]}),$ for each $ m \geq n.$
\end{proof}
\begin{dfn} Let  $T \in L(X)$  be a semi-B-Fredholm operator.  The index of $T$ is defined as the index of  $T_{[n]};$ where $n$ is any integer   such that $\R(T^{n})$ is closed and  $T_{[n]}$ is  semi-Fredholm. From Proposition \ref{propaznayzariouh}, this definition is independent of the choice of the integer $n$ (see also the first Remark given in \cite[p. 459]{berkani-sarih}). Furthermore, if $T$ is a semi-Fredholm operator this reduces to the usual definition of the index.
\end{dfn}
\begin{rema}\label{rema1} Let $T \in L(X)$ be a pseudo-Fredholm and  $(M,N) \in GKD(T).$ Let $n \in \N^{*},$ then $(M,N) \in GKD(T^{n}).$ Hence  $\mathcal{N}(T^{n}_{M})=\mathcal{K}(T)\cap\mathcal{N}(T^{n})$ and $\R(T^{n}_{M})\oplus N=\R(T^{n})+ \mathcal{H}_{0}(T).$ Indeed,  since $T^{n}_{M}$ is semi-regular then \cite[Theorem 1.44, Theorem 1.63]{aiena} imply that $\mathcal{N}(T^{n}_{M})\subset \mathcal{K}(T^{n}_{M})=\mathcal{K}(T^{n})\subset \mathcal{K}(T)=\mathcal{K}(T_{M})\subset M.$ Thus  $\mathcal{N}(T^{n}_{M})=\mathcal{K}(T)\cap \mathcal{N}(T^{n}_{M})=\mathcal{K}(T)\cap M \cap \mathcal{N}(T^{n})=\mathcal{K}(T)\cap \mathcal{N}(T^{n}).$ On the other hand, as  $T_{M}$ is semi-regular then \cite[Corollary 2.38]{aiena} entails that $\mathcal{H}_{0}(T_{M})=T^{n}(\mathcal{H}_{0}(T_{M}))\subset \R(T^{n}_{M}).$ Thus $\R(T^{n})+ \mathcal{H}_{0}(T)=\R(T^{n}_{M})+\R(T^{n}_{N})+ \mathcal{H}_{0}(T_{M})+N=\R(T^{n}_{M})\oplus N.$
Consequently, it is easily seen that  $\R(T^{n}_{M})=\R(T^{m}_{M})$  if and only if $\R(T^{n})+ \mathcal{H}_{0}(T) =\R(T^{m})+ \mathcal{H}_{0}(T),$ for every integers $m, n\in \N^{*}.$ Moreover, $\mbox{codim}_{M}\,\R(T_{M}):=\mbox{dim}\,\frac{M}{\R(T_{M})}=\mbox{dim}\,\frac{X}{\R(T_{M})\oplus N}=\mbox{dim}\,\frac{X}{\R(T)+ \mathcal{H}_{0}(T)}.$
\end{rema}
The previous remark  allows us to introduce the following definition.
\begin{dfn}\label{dfnp} Let $T \in L(X)$ be a pseudo-Fredholm operator, and let $(M,N) \in GKD(T).$ We define the nullity, the  deficiency, the ascent and the descent of $T$  respectively, by $\alpha(T):=\dim\, \mathcal{N}(T_{M}),$  $\beta(T):=\mbox{codim}_{M}\,\R(T_{M}),$ $p(T):=\mbox{inf}\{n\in\N : \mathcal{N}(T^{n}_{M})=\mathcal{N}(T^{m}_{M}) \text{ for all integer } m \geq n\}$ and $q(T):=\mbox{inf}\{n\in\N : \R(T^{n}_{M})=\R(T^{m}_{M}) \text{ for all integer } m \geq n\}.$ From the previous remark,   the nullity, the  deficiency, the ascent and the descent of $T$  are independent of the choice of the generalized kato decomposition $(M,N)$ of $T.$
\end{dfn}
In particular, if $T$ is semi-regular then   $\alpha(T)=\dim\, \mathcal{N}(T)$ and  $\beta(T)=\mbox{codim}\,\R(T).$ And  if $T$ is a B-Fredholm, then  $\R(T^{d})$ is closed, $T_{[d]}$ is semi-regular,   $\alpha(T)=\alpha(T_{[d]})$ and  $\beta(T)=\beta(T_{[d]});$  where $d=\mbox{dis}(T),$ see  Theorem \ref{thm1} below and  \cite[Theorem 1.64]{aiena}.  And if  $T$ is semi-Fredholm, then there exists $n \in \N$ such that   $\alpha(T)=\mbox{dim}\,\mathcal{N}(T)-n$ and   $\beta(T)=\mbox{codim}\,\mathcal{R}(T)-n.$
\par From \cite[Theorem 1.22]{aiena} and Definition \ref{dfnp}, we deduce the relationships between the quantities $\alpha(T),$  $\beta(T),$ $p(T)$ and $q(T).$
\begin{rema}
Let $T$ be a pseudo-Fredholm operator.  We have the following statements.\\
(i) If  $p(T)<\infty$ then $\alpha(T)\leq \beta(T).$\\
(ii) If $q(T)<\infty$ then  $\alpha(T)\geq \beta(T).$\\
(iii) If  $\mbox{max}\,\{p(T), q(T)\}<\infty$ then  $p(T)=q(T)$ and   $\alpha(T)= \beta(T).$\\
(iv)  If $\mbox{max}\,\{\alpha(T), \beta(T)\}<\infty$ and $\mbox{min}\,\{p(T), q(T)\}<\infty,$  then  $p(T)=q(T).$
\end{rema}

\begin{lem} Let $T \in L(X)$ be a semi-regular operator. The following statements hold.\\
(i) If $\alpha(T)<\infty$ then  $T$ is one-to-one if and only if $p(T)<\infty.$\\
(ii)  If $\beta(T)<\infty$ then  $T$ is onto if and only if $q(T)<\infty.$\\
(iii) If $\mbox{max}\,\{\alpha(T), \beta(T)\}<\infty$ then $T$ is bijective  if and only if $p(T)=q(T)<\infty.$
\end{lem}
\begin{proof} (i) Let $n \in \N.$ Since  $T$ is semi-regular then   $T(\mathcal{N}(T^{n+1}))=\mathcal{N}(T^{n})\cap \R(T)=\mathcal{N}(T^{n}),$ and so the operator  $T: \mathcal{N}(T^{n+1})  \longrightarrow  \mathcal{N}(T^{n})$ is onto. As  $\alpha(T)<\infty,$ then  $\alpha(T^{n+1})=\alpha(T)+\alpha(T^{n}).$ By induction we obtain  $\alpha(T^{n})=n.\alpha(T).$ Consequently, $T$ is one-to-one if and only if $p(T)<\infty.$\\
(ii) Suppose that $q(T)<\infty.$  Since $T$ is semi-regular then from  \cite[Theorem 1.43]{aiena},  $T^{*}$ is semi-regular. As $\beta(T)<\infty$ then $\R(T)$ is closed  and $\alpha(T^{*})<\infty.$  From \cite[Lemma 1.26]{aiena} $q(T)=p(T^{*})<\infty$ and this implies by the first point that  $T$ is onto.\\
(iii) Is a direct consequence of the   first and the  second  points.
\end{proof}
 The next proposition gives a characterization of pseudo B-Fredholm  and generalized Drazin invertible operators. We recall \cite{benharrat, koliha}  that $T \in L(X)$ is said to be  left  generalized Drazin invertible (resp., right  generalized Drazin invertible,  generalized Drazin invertible) if $T=T_{1}\oplus T_{2};$ where $T_{1}$ is bounded below (resp., onto, invertible) and $T_{2}$ is quasi-nilpotent.
\begin{prop} Let $T \in L(X).$ The following assertions hold.\\
(i)  $T$ is pseudo B-Fredholm if and only if $T$  is pseudo-Fredholm and $\mbox{sup}\,\{\alpha(T), \beta(T)\}<\infty.$\\
(ii) $T$ is semi pseudo B-Fredholm if and only if $T$  is pseudo-Fredholm and $\mbox{inf}\,\{\alpha(T), \beta(T)\}<\infty.$\\
(iii) $T$ is left   generalized Drazin invertible if and only if $T$ is upper pseudo semi-B-Fredholm and $p(T)<\infty$  if and only if $T$  is pseudo-Fredholm and $p(T)=0.$\\
(iv) $T$ is right  generalized Drazin invertible if and only if $T$ is lower  pseudo semi-B-Fredholm and $q(T)<\infty$  if and only if $T$  is pseudo-Fredholm and $q(T)=0.$\\
(v) $T$ is generalized Drazin invertible if and only if $T$  is pseudo-Fredholm and $p(T)=q(T)<\infty.$
\end{prop}
\begin{proof} We left its proof  as an exercise to the reader.
\end{proof}
{\bf Conjecture.} $T \in L(X)$ is pseudo B-Fredholm if and only if $\mbox{dim}\,\mathcal{K}(T)\cap\mathcal{N}(T)<\infty$ and $\mbox{dim}\,\frac{X}{\R(T)+ \mathcal{H}_{0}(T)}<\infty.$

\vspace{0.2cm}
\par Our next theorem gives a punctured neighborhood theorem for  pseudo semi-B-Fredholm operators. Which in turn it extends   \cite[Theorem 1.117]{aiena} and  \cite[Theorem 3.1]{tajmouati1} by using the notions of nullity, the deficiency and the  index of  pseudo semi-B-Fredholm operator. Hereafter      $\sigma_{se}(T)$  means  the semi-regular spectrum of $T.$
\begin{thm}\label{thm0}  Let $T \in L(X)$ be  a  pseudo semi-B-Fredholm operator, then there exists $\epsilon>0$ such that $B(0, \epsilon)\setminus \{0\} \subset (\sigma_{sf}(T))^C\cap(\sigma_{se}(T))^C.$ Moreover, $\alpha(T)=\alpha(T-\lambda I),$  $\beta(T)=\beta(T-\lambda I)$  and  $\mbox{ind}(T)=\mbox{ind}(T- \lambda I)$ for every $\lambda \in B_{0}.$
\end{thm}
\begin{proof} Since $T$ is  a pseudo semi-B-Fredholm, by Proposition \ref{prop1} we have $T=T_{1}\oplus T_{2};$ where $T_{1}$ is semi-Fredholm  and  semi-regular and $T_{2}$ is quasi-nilpotent. From  the  punctured neighborhood theorem for  semi-Fredholm operators and  \cite[Proposition 4.1]{mbekhta} (which is also valid in Banach spaces, see \cite[Theorem 2.1]{bouamama}), there exists $\epsilon>0$ such that $B_{0}:=B(0, \epsilon)\setminus\{0\} \subset (\sigma_{sf}(T_{1}))^C\cap (\sigma(T_{2}))^C\cap(\sigma_{se}(T))^C\subset(\sigma_{sf}(T))^C\cap(\sigma_{se}(T))^C$ and  $\mbox{ind}(T_{1})=\mbox{ind}(T_{1}- \lambda I)$ for every $\lambda \in B_{0}.$ Since $T_{2}$ is quasi-nilpotent (so is pseudo B-Weyl), then we have  $\mbox{ind}(T_{2})=\mbox{ind}(T_{2}- \lambda I)=0$ for every $\lambda \in \C.$ Moreover,  from the point (iii) of Remark \ref{rema0}, we conclude that   $\mbox{ind}(T)=\mbox{ind}(T- \lambda I)$ for every $\lambda \in B_{0}.$ On  the other hand,  Proposition \ref{prop1} and  \cite[Theorem 1.60]{aiena} imply that there exists $\epsilon>0$ such that $\alpha(T)=\alpha(T-\lambda I)$ and $\beta(T)=\beta(T-\lambda I)$ for all $\lambda \in B_{0}.$
\end{proof}

The next corollary is a consequence of Theorem \ref{thm0}.
\begin{cor}\label{cor1}  Let $T \in L(X).$ Then\\
(i) $\sigma_{upbf}(T),$ $\sigma_{lpbf}(T),$ $\sigma_{spbf}(T),$ $\sigma_{pbf}(T),$ $\sigma_{upbw}(T),$ $\sigma_{lpbw}(T),$ $\sigma_{spbw}(T)$ and  $\sigma_{pbw}(T)$  are  a compact subsets of $\C.$\\
(ii) If $\Omega$ is a component of $(\sigma_{upbf}(T))^C$ or of $(\sigma_{lpbf}(T))^C,$ then the index $\mbox{ind}(T -\lambda I)$ is constant as $\lambda$ ranges over $\Omega.$
\end{cor}

For proving \cite[Theorem 2.4]{berkani-castro},  the authors used the characterization of B-Fredholm  operators in the case of Hilbert spaces based on the Kato's decomposition of quasi-Fredholm operators \cite{labrousse}.
As an  application of   Lemma  \ref{lem0}, we extend   \cite[Theorem 2.4]{berkani-castro} to the general case of  Banach space. Precisely, we shows that if $T \in L(X)$  is B-Fredholm, then the index of $T$ as a B-Fredholm operator coincides with  its index as a pseudo B-Fredholm.

  In the sequel,  for a  semi-B-Fredholm operator  $T \in L(X),$   we take $n$ an   integer  such that $\R(T^{n})$ is closed and  $T_{[n]}$ is  semi-Fredholm.
\begin{thm}\label{thm1}
$T \in  L(X)$ is  a B-Fredholm operator if and only if $T = T_1 \oplus T_2$ such that  $T_1$ is  Fredholm and semi-regular and $T_2$ is nilpotent.
Moreover,  in this case $T$ is pseudo B-Fredholm  and $\alpha(T)=\alpha(T_{1}),$  $\beta(T)=\beta(T_{1}),$  $p(T)=p(T_{1}),$    $q(T)=q(T_{1})$  and $\mbox{ind} (T) = \mbox{ind}(T_1)=\mbox{ind} (T_{[n]}).$
\end{thm}
\begin{proof}
From     \cite[Theorem 7]{muller} and  \cite[Theorem 4]{kato},  we have  $T$ is B-Fredholm  if and only if  $T = T_1 \oplus T_2;$ where   $T_1$ is  Fredholm and semi-regular and $T_2$ is nilpotent. Thus, every B-Fredholm operator is a pseudo B-Fredholm operator. Let us  to show that if such a decomposition exists for $T,$ then $\mbox{ind}(T) = \mbox{ind} (T_1)=\mbox{ind} (T_{[n]}).$\\
Suppose that $T = T_1 \oplus T_2$ is  a B-Fredholm operator. Since $T_1$ is Fredholm, then  by  Theorem \ref{thm0} and \cite[Corollary 4.7]{Grabiner},   there exists $\epsilon >0$ such that $B_{0}:=B(0, \epsilon)\setminus \{0\} \subset  (\sigma_{e}(T))^C,$  $\mbox{ind}(T_1 - \lambda I) =\mbox{ind}(T_1)$ and  $\mbox{ind}(T - \lambda I) =\mbox{ind}(T)=\mbox{ind}(T_{[n]})$ for all $\lambda \in B_{0}.$  Let $\lambda \in B_{0},$ as  $T_2$ is nilpotent, then $\lambda \notin \sigma(T_2).$ Hence $T-\lambda I = T_1-\lambda I \oplus T_2 - \lambda I$ is a Fredholm operator and $\mbox{ind}(T - \lambda I)=\mbox{ind}(T_1 - \lambda I).$ Consequently, $\mbox{ind} (T) = \mbox{ind}(T_1)=\mbox{ind} (T_{[n]}).$ The rest of the proof comes from the Remark \ref{rema1}.
\end{proof}

We don't know if Theorem \ref{thm1}  can be extended to the  case of semi-B-Fredholm operators.  Whereas  the following proposition gives  a version of Theorem \ref{thm1} for  semi-B-Fredholm operators.  Note that in the case of $X$  is a Hilbert space, it is proved in \cite[Theorem 2.6]{berkani-sarih} that  $T$  is a semi-B-Fredholm operator if and only if $T = T_1 \oplus T_2$ such that  $T_1$ is  semi-Fredholm and $T_2$ is nilpotent. In this case  and if    $T$ is an upper   semi-B-Fredholm operator, then $\mbox{ind} (T_{1})=\mbox{ind} (T),$ see \cite[Proposition 2.9]{berkani-koliha}.
\begin{prop}\label{prop2}
$T \in L(X)$ is  a  semi-B-Fredholm  and   pseudo-Fredholm if and only if  $T$  is a direct sum  of a  semi-Fredholm operator and a nilpotent operator.  Moreover, the index of $T$ as a semi-B-Fredholm coincides with its index as a  pseudo semi-B-Fredholm.
\end{prop}
\begin{proof} Let $(M, N)\in GKD(T),$ then $T_{M}$ is semi-regular  so that  $(T_{M})_{[m]}$ is also  semi-regular and $\alpha(T_{M})=\alpha((T_{M})_{[m]})$ and $\beta(T_{M})=\beta((T_{M})_{[m]})$  for every $m\in \N.$ Since $T$ is semi-B-Fredholm  then  $T_{M}$ is  semi-B-Fredholm, and so $T_{M}$ is  a  semi-Fredholm  operator. On the other hand, $T_{N}$ is  Drazin invertible, since $T$ is semi-B-Fredholm. Hence $T_{N}$ is  nilpotent.
\par Conversely,  let $(M, N) \in Red(T)$ such that $T_{M}$ is semi-Fredholm and $T_{N}$ is nilpotent. Then there exists $(A, B) \in Red(T_{M})$ such that $T_{A}$ is semi-Fredholm and semi-regular  and $T_{B}$ is nilpotent. So $(A, B\oplus N) \in Red(T)$ and $T_{B\oplus N}$ is nilpotent of degree $d.$ Hence $\R(T^{d})=\R(T^{d}_{A})$ and  $\mathcal{N}(T_{A})=\mathcal{N}(T_{[d]})$ and $T(A)\oplus (B\oplus N)=\mathcal{N}(T^{d})+\R(T).$ Therefore $\alpha(T_{A})=\alpha(T_{[d]})$ and $\beta(T_{A})=\beta(T_{[d]}).$  So  $\R(T^{d})$ is closed  and $T_{[d]}$ is semi-Fredholm, and then $T$ is   semi-B-Fredholm.   Moreover,  $\mbox{ind} (T) =\mbox{ind} (T_{[d]})=\mbox{ind} (T_{[n]}).$
\end{proof}
The following corollary \ref{cor2}  shows that  Theorem \ref{thm1} can be extended to  semi-B-Fredholm operators  in the case of $X$ is a Hilbert space, since every closed subspace of a Hilbert space is complemented. Before that we recall some basic definitions which will be needed later.
\begin{dfn}\cite{kato, mbekhta-muller}  Let $T\in  L(X).$\\
(i) $T$ is called a quasi-Fredholm operator of degree $d$ if   $d=\mbox{dis}(T) \in \N$ and $R(T^{d+1})$ is closed.\\
(ii) We said that   $T$ is decomposable in the  Kato's sense of degree $d$ if  there exists  $(M, N)\in Red(T)$  such that  $T_{M}$ is semi-regular     and $T_{N}$ is nilpotent of  degree $d.$
\end{dfn}
It is well known \cite{labrousse} that  the degree $d$  of a decomposable operator $T\in L(X)$  in the   Kato's sense  is  well defined.
\begin{cor}\label{cor2}  Let $T \in L(X)$ be an upper  semi-B-Fredholm  (resp., a lower  semi-B-Fredholm) operator    such that  $\R(T)+\mathcal{N}(T^{d})$  (resp.,  $\R(T^{d})\cap \mathcal{N}(T)$) has a complementary in $X;$ where $d=\mbox{dis}(T).$ Then $T$ is pseudo semi-B-Fredholm and $\mbox{ind} (T)=\mbox{ind} (T_{[n]}).$
\end{cor}
\begin{proof} If   $T \in L(X)$ is  an upper  semi-B-Fredholm then  from \cite[Proposition 2.5]{berkani-sarih}, $T$ is quasi-Fredholm operator  of degree $d$ and the subspace  $\mathcal{N}(T_{[d]})=\R(T^{d})\cap \mathcal{N}(T)$ is  of finite  dimension. If $d=0$ then $T$ is an upper semi-Fredholm, since $\R(T^{d})$ is closed and $T_{[d]}$ is  upper semi-Fredholm. Thus,   $T$ is a pseudo semi-B-Fredholm operator. Suppose that $d>0,$ by assumption   we have      $\R(T)+\mathcal{N}(T^{d})$ is complemented. So  $T$ is decomposable in the   Kato's sense of degree $d$   (see  \cite[Remark p. 206]{labrousse}), that's  there exists  $(M, N)\in Red(T)$  such that  $T_{M}$ is semi-regular     and $T_{N}^{d}=0.$  Thus $\R(T_{M})$ is closed,  $\alpha(T_{[d]})=\alpha((T_{M})_{[d]})=\alpha(T_{M})<\infty$ and $\beta(T_{[d]})=\beta((T_{M})_{[d]})=\beta(T_{M}).$    Hence $T_{M}$ is an upper semi-Fredholm operator.  By Proposition $\ref{prop2}$ we deduce the  desired result. The case of $T$ is a lower semi-B-Fredholm operator with $\R(T^{d})\cap \mathcal{N}(T)$  has a complementary goes similarly.
\end{proof}
  Let $M$ be a subset of  $X$ and  $N$ a subset of $X^{*}.$ The annihilator of $M$ and the pre-annihilator of  $N$  are  the closed subspaces  defined respectively, by
$$M^{\perp} := \{f \in X^{*} : f (x) = 0 \text{ for every }  x \in M\},$$
and
$${}^{\perp}N := \{x \in X : f (x) = 0 \text{ for every } f \in N\}.$$
Let $T \in L(X)$ and let $(M,N)\in Red(T),$   we denote by  $P_{M}$  the projection on $M$  according to the decomposition of $X=M\oplus N.$
\begin{lem}\label{lem12} Let $T \in L(X)$ and let  $(M,N) \in Red(T)$ such that  $\R(T_{M})\oplus N$ is closed, then $\R(T^{*}_{N^{\perp}})\oplus M^{\perp}$ is closed in  $\sigma(X^{*}, X);$ where $\sigma(X^{*}, X)$ is the weak*-topology on $X^{*}.$
\end{lem}
\begin{proof}
Suppose that $\R(T_{M})\oplus N$ is closed, and let $\overline{T} \in L(\frac{X}{\mathcal{N}(T_{M})}, \R(T_{M})\oplus N)$ the operator defined  by $\overline{T}(\overline{x})=T(P_{M}(x))+ P_{N}(x).$ It is easily seen that $\overline{T}$ is well defined and is  an isomorphism. On the  other hand, as $\mathcal{N}(T_{M})^{\perp}=\overline{\R(T^{*}_{N^{\perp}})\oplus M^{\perp}}^{\sigma(X^{*}, X)}$ then it suffices to show that  $\mathcal{N}(T_{M})^{\perp}\subset \R(T^{*}_{N^{\perp}})\oplus M^{\perp}.$ Let $f\in \mathcal{N}(T_{M})^{\perp}$ and let  $\overline{f}\in  L(\frac{X}{\mathcal{N}(T_{M})},\C)$ the linear form defined by  $\overline{f}(\overline{x})=f(x).$ Let $g\in X^{*}$ be the extension of  $\overline{f}(\overline{T})^{-1}$  given by the Hahn-Banach theorem. Hence $f=T^{*}(g)+f(I-T)P_{M} \in   \R(T^{*})+ M^{\perp}=\R(T^{*}_{N^{\perp}})\oplus M^{\perp}.$
\end{proof}
Let  $X$ be a Banach space, it is well known that $\mbox{dim}X \leq \dim X^{*};$ where $X^{*}$ is the  topological dual. In the next theorem we  do not distinguish between $\mbox{dim}X$ and $\dim X^{*},$ that is if $\dim X= \infty$  then we write  $\dim X=\dim X^{*}=\infty.$\\
The proof of the next theorem  is based on  the  classical  theorems  \cite[Theorem A.1.8,  Theorem A. 1.9]{laursen-neumann}.
\begin{thm}\label{thm24}  If $T \in L(X)$ is   a    pseudo-Fredholm then   $T^{*}$  is    pseudo-Fredholm, $\alpha(T)=\beta(T^{*}),$  $\beta(T)=\alpha(T^{*}),$ $p(T)=q(T^{*})$ and $q(T)=p(T^{*}).$ In particular, if $T$ is pseudo semi-B-Fredholm  then $T^{*}$ is pseudo semi-B-Fredholm and     $\mbox{ind}(T)=-\mbox{ind}(T^{*}).$
\end{thm}
\begin{proof}
 Suppose that $T$  is a  pseudo-Fredholm. Then  there exists $(M, N) \in GKD(T).$  From (which is also true in the Banach spaces) \cite[Theorem 3.3]{mbekhta} that   $(N^{\perp}, M^{\perp})\in GKD(T^{*}).$  Let $n \in \N,$   then  $\mathcal{N}((T^{*}_{N^{\perp}})^{n})=(N+\R(T^{n}))^{\perp}=(N\oplus \R(T_{M}^{n}))^{\perp}.$  As  $N\oplus \R(T_{M})$ is closed then $\alpha(T^{*})=\alpha(T^{*}_{N^{\perp}})=\mbox{dim}\,(\frac{X}{N\oplus\R(T_{M})})^{*}=\mbox{dim}\,(\frac{M}{\R(T_{M})})^{*}=\beta(T_{M})=\beta(T),$  and from Remark \ref{rema1} we then obtain $p(T^{*})=p(T^{*}_{N^{\perp}})=q(T_{M})=q(T).$ Since  $\mathcal{N}(T^{n})={}^{\perp}\R((T^{*})^{n})$  then $\mathcal{N}(T_{M}^{n})={}^{\perp}(\R((T^{*})^{n})+M^{\perp})={}^{\perp}(\R((T^{*}_{N^{\perp}})^{n})\oplus M^{\perp}).$  Using again  Remark \ref{rema1} we deduce $p(T)=q(T^{*}).$ On the other hand, the previous lemma shows that  $\R(T^{*}_{N^{\perp}})\oplus M^{\perp}$ is   closed in  $\sigma(X^{*}, X)$ and   then     $(\mathcal{N}(T_{M})^{*}) \cong  \frac{X^{*}}{M^{\perp}\oplus \R(T^{*}_{N^{\perp}})}\cong \frac{N^{\perp}}{\R(T^{*}_{N^{\perp}})}.$ Thus  $\alpha(T)=\alpha(T_{M})=\mbox{dim}\,(\mathcal{N}(T_{M})^{*})=\beta(T^{*}_{N^{\perp}})=\beta(T^{*}).$ Consequently,  if   $T_{M}$ is  semi-Fredholm  then   $T^{*}_{N^{\perp}}$ is semi-Fredholm  and $\mbox{ind}(T)=\mbox{ind}(T_{M})=\alpha(T_{M})-\beta(T_{M})=\beta(T_{N^{\perp}}^{*})-\alpha(T_{N^{\perp}}^{*})=-\mbox{ind}(T^{*}_{N^{\perp}})=-\mbox{ind}(T^{*}).$
\end{proof}
From the previous results, we obtain  the next corollary.
\begin{cor} If  $T\in L(X)$ is  a B-Fredholm operator,  then $\R(T^{*})+\mathcal{N}(T^{*d})$ is closed in $\sigma(X^{*}, X)$ and $T^{*}$ is  a  B-Fredholm operator  with    $\alpha(T_{[d]})=\beta(T^{*}_{[d]}),$  $\beta(T_{[d]})=\alpha(T^{*}_{[d]})$  and  $\mbox{ind}(T)=\mbox{ind}(T_{[d]})=-\mbox{ind}(T^{*}_{[d]})=-\mbox{ind}(T^{*});$ where $d=\mbox{dis}(T).$
\end{cor}

\begin{proof}
We know from \cite[Theorem 7]{muller} that  $T$ is decomposable in the Kato's sense of degree $d^{'}.$  More precisely, there exists $(M, N) \in  Red(T)$ such that  $T_{M}$ is semi-regular, $T_{N}$ is nilpotent of degree $d^{'}$ and  $N \subset \mathcal{N}(T^{d}).$ Thus   $N=\mathcal{N}(T_{N}^{d})=\mathcal{N}(T_{N}^{d^{'}})$ and then $d\geq d^{'}.$  Let us to show that $d=d^{'}.$   Let $m\geq d^{'},$ since $T_{M}$ is semi-regular then  $\R(T)+\mathcal{N}(T^{m})=\R(T_{N})+\mathcal{N}(T_{N}^{m})+\R(T_{M})+\mathcal{N}(T_{M}^{m})=N\oplus \R(T_{M})=\R(T_{N})+\mathcal{N}(T_{N}^{d})+\R(T_{M})+\mathcal{N}(T_{M}^{d})=\R(T)+\mathcal{N}(T^{d}).$ Hence $d=d^{'}$ and       $\R(T_{M})\oplus N = \R(T)+\mathcal{N}(T^{d})$ is closed.   On the other hand, it is well known that $T^{*}$  is decomposable in the Kato's sense of degree   $\mbox{dis}(T^{*})=d,$ and from Theorem \ref{thm24} we have  $\alpha(T_{[d]})=\beta(T^{*})$ and   $\beta(T_{[d]})=\alpha(T^{*}).$  Hence   $T^{*}$ is a    B-Fredholm,  and  from  Lemma \ref{lem12} and  what precedes we have  $\R(T^{*}_{N^{\perp}})\oplus M^{\perp}=\R(T^{*})+\mathcal{N}(T^{*d})$ is  closed in $\sigma(X^{*}, X).$  Consequently,   $\alpha(T_{[d]})=\beta(T^{*}_{[d]}),$  $\beta(T_{[d]})=\alpha(T^{*}_{[d]})$  and  $\mbox{ind}(T)=\mbox{ind}(T_{[d]})=-\mbox{ind}(T^{*}_{[d]})=-\mbox{ind}(T^{*}).$
\end{proof}

\goodbreak

{\small \noindent Zakariae Aznay,\\  Laboratory (L.A.N.O), Department of Mathematics,\\Faculty of Science, Mohammed I University,\\  Oujda 60000 Morocco.\\
aznay.zakariae@ump.ac.ma\\

{\small \noindent  Abdelmalek Ouahab,\\  Laboratory (L.A.N.O), Department of Mathematics,\\Faculty of Science, Mohammed I University,\\  Oujda 60000 Morocco.\\
ouahab05@yahoo.fr\\

 \noindent Hassan  Zariouh,\newline Department of
Mathematics (CRMEFO),\newline
 \noindent and laboratory (L.A.N.O), Faculty of Science,\newline
  Mohammed I University, Oujda 60000 Morocco.\\
 \noindent h.zariouh@yahoo.fr


\begin{thebibliography}{20}

\bibitem{aiena} P. Aiena, {\it Fredholm and Local Spectral Theory II, with Application to Weyl-type Theorems},
Springer Lecture Notes of Math no. 2235, (2018).

\bibitem{amouch} M.  Amouch,  M.  Karmouni  and  A.  Tajmouati, {\it Spectra  originated  from  Fredholm  theory  and  Browder's  theorem},  Commun. Korean Math. Soc., {\bf 33} (2018),  853--869.


\bibitem {berkani} M. Berkani,  {\it On a class of quasi-Fredholm operators}, Integr. Equ. and Oper. Theory,  {\bf 34} (1999), 244--249.

\bibitem{berkani-castro} M. Berkani and N. Castro, {\it Unbounded B-Fredholm operators on Hilbert spaces}, Proc. Edinb. Math. Soc., {\bf 51} (2008),   285--296.

\bibitem{berkani-koliha} M. Berkani and  J. J. Koliha, {\it Weyl type theorems for bounded linear operators}, Acta Sci. Math. (Szeged), {\bf 69} (2003), 359--376.


\bibitem{berkani-sarih} M. Berkani and M. Sarih, {\it On semi-B-Fredholm operators},   Glasg. Math. J., {\bf 43} (2001), 457--465.


\bibitem{boasso} E. Boasso, {\it Isolated spectral points and Koliha-Drazin invertible elements in quotient Banach algebras and homomorphism ranges}, Mathematical Proceedings of the Royal Irish Academy, 115A, (2015), 1--15.

\bibitem{rwassa} M. D. Cvetkovi\'{c},  S\v{C}. \v{Z}ivkovi\'{c}-Zlatanovi\'{c},  {\it Generalized Kato decomposition and essential spectra}, Complex Anal. Oper. Theory, {\bf 11} (2017), 1425--1449.

\bibitem{bouamama} W. Bouamama, {\it Op\'erateurs    Pseudo    Fredholm    dans    les    espaces    de Banach}, Rend. Circ. Mat. Parelmo, (2004), 313--324.

\bibitem{Grabiner}  S. Grabiner, {\it Uniform ascent and descent of bounded operators},  J. Math. Soc. Japan,  {\bf 34} (1982), 317--337.

\bibitem{benharrat}  K. Hocine, M. Benharrat, B. Messirdi, {\it Left and right generalized Drazin invertible operators}, Linear and Multilinear Algebra, {\bf 63} (2015), 1635--1648.

\bibitem{kaashoek} M. A. Kaashoek, {\it Ascent, Descent, Nullity and Defect, a Note on a Paper by A. E. Taylor},  Math. Annalen,  {\bf 172} (1967), 105--115.

\bibitem{kato} T. Kato, {\it Perturbation theory for nullity, deficiency and other quantities of linear operators}, J. Anal. Math. {\bf 6} (1958), 261--322.

\bibitem{koliha}  J. Koliha, {\it  A generalized Drazin inverse. Glasgow Mathematical Journal}, {\bf 38} (1996), 367--381.



\bibitem{labrousse} J. P. Labrousse, {\it Les op\'erateurs quasi-Fredholm: une g\'en\'eralisation des op\'erateurs semi Fredholm}, Rend. Circ. Math. Palermo, {\bf 29} (1980), 161--258.

\bibitem{laursen-neumann} K. B. Laursen., M. M.  Neumann, {\it An Introduction to Local Spectral Theory}, Clarendon Press, Oxford, (2000).

\bibitem{mbekhta}  M. Mbekhta, {\it Op\'erateurs pseudo-Fredholm I: R\'esolvant g\'en\'eralis\'e,} J. Operator Theory, {\bf 24}
(1990), 255--276.

\bibitem{mbekhta-muller} M. Mbekhta, V. M\"{u}ller, {\it On the axiomatic theory of spectrum II}, Studia Math., {\bf 119} (1996), 129--147.


\bibitem{muller} V. M\"{u}ller,  {\it On the Kato-decomposition of quasi-Fredholm and B-Fredholm operators}, Vienna, Preprint ESI, 1013 (2001).

\bibitem{poon} P. W. Poon, {\it On the stability radius of a quasi-Fredholm operator}, Proc. Amer. Math. Soc., {\bf126} (1998) 1071--1080.

 \bibitem{tajmouati1} A. Tajmouati, M. Karmouni,  M. Abkari, {\it Pseudo semi-B-Fredholm and Generalized
Drazin invertible operators Through Localized SVEP}, Italian Journal of pure and applied
mathematics, {\bf 37} (2017), 301--314.

\bibitem{taylor} A. E. Taylor, {\it Theorems on Ascent, Descent, Nullity and Defect of Linear Operators}, Math. Annalen, {\bf 163}   (1966), 18--49.

\bibitem{zariouh-zguitti} H. Zariouh, H. Zguitti, {\it On pseudo B-Weyl operators and generalized drazin invertible for operator matrices}, Linear and Multilinear Algebra, {\bf 64} (2016), 1245--1257.



\end{thebibliography}
\end{document}